\documentclass[10pt]{amsart}
\usepackage[utf8]{inputenc}
\usepackage[T1]{fontenc}
\usepackage{float}
\usepackage{latexsym}
\usepackage{amssymb}
\usepackage{amsmath}
\usepackage{tikz}
\usetikzlibrary{calc}
\usepackage{lmodern}
\usepackage{amsthm}
\usepackage{enumerate}
\usepackage{cite}
\newtheorem{theorem}{\textsc{Theorem}}[section]
\newtheorem{proposition}[theorem]{\textsc{Proposition}}
\newtheorem{corollary}[theorem]{\textsc{Corollary}}

\theoremstyle{definition}
\newtheorem{definition}[theorem]{\textsc{Definition}}

\theoremstyle{remark}
\newtheorem{remark}[theorem]{\textsc{Remark}}
\author{J. Kwon}
\date{\today}
\begin{document}
\title{A Universal Bound on the Variations of Bounded Convex Functions}
\subjclass[2010]{26B25, 52A05}
\keywords{Convex Functions, Variations, Funk Metric, Thompson Metric,
  Hilbert Metric}

% fonction convexe bornée définie sur R.

\thanks{The author is grateful to Pierre-Antoine Corre, Rida Laraki, Sylvain Sorin and Guillaume Vigeral for their very helpful comments.}

\begin{abstract}
% We provide bounds on the variations of a bounded convex function
% between two given points. We establish links with the Funk, Thompson
% and Hilbert metrics, and also show that the bound is optimal. We then
% investigate the implied restriction on the subdifferential.
  
Given a convex set $C$ in a real vector space $E$ and two points $x,y\in C$, we investivate which are the possible
values for the variation $f(y)-f(x)$, where $f:C\longrightarrow [m,M]$ is a bounded
convex function. We then rewrite the bounds in terms of the Funk weak
metric, which will imply that a bounded convex function is
Lipschitz-continuous with respect to the Thompson and Hilbert metrics.
The bounds are also proved to be optimal. We also exhibit the
maximal subdifferential of a bounded convex function at a given point $x\in C$.
\end{abstract}
\maketitle
\section{The Variations of Bounded Convex Functions}
\label{sec-2}
\label{SEC:variations}
Let $C$ be a convex set of a real vector space $E$. Given two points $x,y\in C$, we define the
following auxiliary quantity:
\[ \tau_C(x,y)=\sup_{}\left\{ t\geqslant 1\,|\,x+t(y-x)\in C \right\}. \]
Clearly, $\tau_C$ takes values in $[1,+\infty]$. Intuitively, it measures how far away $x$ is from the boundary in the
direction of $y$, taking the ``distance'' $xy$ as unit. Clearly,
$\tau_C(x,y)=+\infty$ if and only if $x+\mathbb{R}_+(y-x)\subset
C$. Our first result is the following.

% \begin{figure}[h]
%   \begin{center}
%     \begin{tikzpicture}[scale=1.5]
%       \draw (-.5,0) -- (0,0)node{$\bullet$}node[below]{$x$}--
%       (1,0)node{$\bullet$}node[below]{$y$} --
%       (3,0) node[above right]{$\partial C$};
% \draw (15:3) arc (15:-15:3);
% \draw [<->] (0,.25) -- (0.5,.25)node[above]{$1$} --(1,.25); 
% \draw [<->] (0,-.33) -- (1.5,-.33)node[below]{$\tau_C (x,y)$} --
% (3,-.33);
% \draw  (2,.5)node{$C$};

%     \end{tikzpicture}
%   \end{center}
% \caption{An intuitive representation of what $\tau_C (x,y)$ measures.}
% \label{fig:tau}
% \end{figure}

\begin{theorem}
\label{thm:bound}
Let $m\leqslant M$ be two real numbers. Let $C$ be a convex set of a
real vector space $E$ and
$f:C\longrightarrow [m,M]$ a convex function. For every couple of
points $(x,y)\in C^2$, $f$ satisfies:
\[ -\frac{M-m}{\tau_C(y,x)}\leqslant f(y)-f(x)\leqslant \frac{M-m}{\tau_C(x,y)}. \]
\end{theorem}
% \begin{figure}[htbp]
% \label{fig:bound}
% \centering
% \begin{tikzpicture}[scale=1.3]
% \draw [<->] (0,3.5) 
%         |- (4.5,0);
% \draw [ultra thick] (0,0) -- (4,0);
% \node[left] at (0,0) {$0$};
% \node[left] at (0,3) {$1$};
% \draw[dotted] (0,3) -- (4,3);
% \draw[dotted] (4,0) -- (4,3);
% \node at (2.25,0){$\bullet$};
% \node at (1,0){$\bullet$};
% \node [below] at (2.25,0){$y$};
% \node [below] at (1,0){$x$};
% \draw (1,0) -- (4,3);
% \draw (2.25,0) -- (2.25,1.25) node [midway,right] {$\tau_C(x,y)^{-1}$};
% \draw [<->] (1,-0.4) -- (2.25,-.4) node [midway,above] {$1$};
% \draw [<->] (1,-.6) -- (4,-.6) node [midway,below] {$\tau_C(x,y)$};
% \end{tikzpicture}
% \caption{Illustration of the bound in the case $m=0$ and $M=1$. The thick horizontal line represents the cross section of $C$.}
% \end{figure}
\begin{proof}
It is enough to prove the result for functions with values in $[0,1]$,
since we can consider $(M-m)^{-1}(f-m)$. Let $x,y$ be two points in $C$. Let $t$ be such
that $1\leqslant t< \tau_C(x,y)$. By definition of $\tau_C$, and
because $C$ is convex, we have $x+t(y-x)\in C$. We can write
$y$ as a convex combination of $x+t(y-x)$ and $x$ with coefficients $1/t$ and $(t-1)/t$ respectively:
\[ y=\frac{x+t(y-x)+(t-1)x}{t}. \]
By convexity of $f$, we get:
\begin{align*}
f(y)-f(x)&\leqslant \frac{f(x+t(y-x))+(t-1)f(x)}{t}-f(x)\\
&\leqslant \frac{f(x+t(y-x))-f(x)}{t}\leqslant \frac{1}{t},
\end{align*}
where the last inequality comes from the fact that $f$ has values in $[0,1]$. By taking the limit as $t\rightarrow \tau_C(x,y)$, we get:
\[ f(y)-f(x)\leqslant \frac{1}{\tau_C(x,y)}. \]
The lower bound is obtained by exchanging the roles of $x$ and $y$.
\end{proof}
\section{The Funk, Thompson and Hilbert Metrics}
\label{sec-3}
\label{SEC:metrics}
In this section, we rewrite the result from Theorem ~\ref{thm:bound} as a
Lipschitz-like property in the framework of convex sets in 
normed spaces. But $1/\tau_C$ is far from being a distance. We thus
consider the Funk, Thompson and
Hilbert metrics (which were introduced in
\cite{funk1929geometrien}, \cite{thompson1963certain} and \cite{hilbert1895gerade} respectively) and establish the link with $\tau_C$.

We restrict our framework to the case where $C$ is an open convex subset of a normed space $(E,\left\|\ \cdot\ 
\right\| )$. Let $x,y\in C$. If $\tau_C(x,y)<+\infty$, we can 
define $b(x,y)$ to be the following point:
\[ b(x,y)=x+\tau_C(x,y)(y-x). \]
Note that since $C$ is open, when $b(x,y)$ exists, it is necessarily
different from $y$. This will be necessary to state the following definitions.

\begin{definition}
Let $C$ be an open convex subset of a normed space $(E,\left\| \ \cdot\  \right\|)$. We define

\begin{enumerate}[(i)]
\item the Funk weak metric:
\[ F_C(x,y)=
\begin{cases}
\displaystyle \log \frac{\displaystyle \left\| x-b(x,y) \right\|}{\displaystyle \left\| y-b(x,y) \right\|}&\text{if $\tau_C(x,y)<+\infty$}\\
0&\text{otherwise}
\end{cases};
 \]
\item the Thompson pseudometric:
\[ T_C(x,y)=\max_{}\left( F_C(x,y),F_C(y,x) \right); \]
\item the Hilbert pseudometric:
\[ H_C(x,y)=\frac{1}{2}\left( F_C(x,y)+F_C(y,x) \right). \]
\end{enumerate}
\end{definition}

\begin{remark}
Even if we will abusively call them \emph{metrics}, they fail to
satisfy the separation axiom in general. The Thompson and the Hilbert
metrics are thus \emph{pseudometrics}. Moreover, the Funk metric not being
symmetric, it actually is a \emph{weak} metric. The Thompson and the Hilbert metrics are respectively the
\emph{max-symmetrization} and \emph{meanvalue-symmetrisation} of the Funk metric. For a detailed presentation of these notions, see e.g. \cite{papadopoulos2007weak}.
\end{remark}

We now establish the link between $\tau_C(x,y)$ and $F_C(x,y)$.
\begin{proposition}
\label{prop:a}
Let $C$ be an open convex subset of a normed space $(E,\| \cdot\ \|)$. For every points $x,y\in C$, the following equality holds:
\[ F_C(x,y)=-\log \left( 1-\frac{1}{\tau_C(x,y)} \right). \]
\end{proposition}
\begin{proof}
Let $x,y\in C$. If $\tau_C(x,y)=+\infty$, the right-hand side of the
above equality is zero, as expected. If $\tau_C(x,y)<+\infty$,
$\tau_C(x,y)$ can be expressed with the norm. Since by definition
$b(x,y)=x+\tau_C(x,y)(y-x)$, we have
\[ \tau_C(x,y)=\frac{\left\| x-b(x,y) \right\|}{\left\| x-y \right\|}\quad\text{and}\quad \tau_C(x,y)-1=\frac{\left\|y-b(x,y)\right\|}{\left\| x-y \right\|}.\]
And thus:
\[
 \frac{ \left\| x-b(x,y) \right\|}{ \left\| y-b(x,y) \right\|}= \left( 1- \frac{1}{\tau_C (x,y)} \right)^{-1}.
\]
Therefore,
\[ F_C(x,y)=-\log\left( 1-\frac{1}{\tau_C(x,y)}\right). \]
\end{proof}

By combining Theorem~\ref{thm:bound} and the above proposition, we
get the following corollary.

\begin{corollary}
\label{cor:a}
Let $C$ an open convex subset of a normed space $(E,\left\| \ \cdot\
\right\| )$ and $f:C\longrightarrow [m,M]$ be a convex function. Then, for all $x,y\in C$, the following bounds hold.
\begin{enumerate}[(i)]
\item $\displaystyle  -(M-m)\left(1-e^{-F_C(y,x)}\right)\leqslant f(y)-f(x)\leqslant (M-m)\left(1-e^{-F_C(x,y)}\right). $

\item $\displaystyle  \left| f(y)-f(x) \right| \leqslant (M-m)\left( 1-e^{-T_C(x,y)} \right). $

\item $\displaystyle  \left| f(y)-f(x) \right| \leqslant (M-m)\left( 1-e^{-2H_C(x,y)} \right). $
\end{enumerate}
\end{corollary}

\begin{remark}
From (ii), by using the inequality $e^{-s}\geqslant 1-s$, we get:
\begin{align*}
\left| f(x)-f(y) \right|&\leqslant (M-m)\left( 1-e^{-T_C(x,y)} \right)\\
&\leqslant (M-m)T_C(x,y),
\end{align*}
and similarly for (iii). Every convex function $f:C\longrightarrow [m,M]$ is thus $(M-m)$-Lipschitz (resp. $2(M-m)$-Lipschitz) with respect to the Thompson metric (resp. the Hilbert metric).
\end{remark}
\section{Optimality of the Bounds}
\label{sec-4}
\label{SEC:tight}
We show in this section that the bounds obtained in Theorem~\ref{thm:bound} are
optimal in the following sense. For a given convex set, and for a given
couple a points, there is a function which attains the upper bound
(resp. the lower bound). In other words, for $x,y\in C$:
\begin{equation*}
\begin{cases}
\displaystyle \max_{\substack{f:C\longrightarrow [m,M]\\f\text{ convex}}}\left( f(y)-f(x) \right)=\frac{M-m}{\tau_C(x,y)}\\
\displaystyle \min_{\substack{f:C\longrightarrow [m,M]\\f\text{ convex}}}\left( f(y)-f(x) \right)=-\frac{M-m}{\tau_C(y,x)}.
\end{cases}
\end{equation*}

In the proof of the following theorem, it will be very convenient to extend the
notion of convexity to functions defined on $C$ and taking values in $\mathbb{R}\cup \left\{-\infty
\right\}$ (and not $\mathbb{R}\cup \{+\infty\}$).
Obviously, the
result according to which the upper envelope of two convex functions
is also a convex function remains true.

\begin{theorem}
\label{thm:tight}
Let $m\leqslant M$ be two real numbers. Let $C$ be a convex set of a
real vector space $E$. For every couple of points $(x,y)\in C^2$, there exists a convex function $f:C\longrightarrow [m,M]$ (resp. $g:C\longrightarrow [m,M]$) such that the upper bound (resp. lower bound) of Theorem~\ref{thm:bound} is attained; in other words:
\[ f(y)-f(x)=\frac{M-m}{\tau_C(x,y)}\qquad \left( \text{resp.}\  g(y)-g(x)=-\frac{M-m}{\tau_C(y,x)} \right).\]
\end{theorem}
\begin{proof}
Let $x$ and $y$ be two points in $C$, and let us construct a convex
function $f:C\longrightarrow [0,1]$ satisfying the equality. If $\tau_C(x,y)=+\infty$, the bound is zero, and $f=0$ is adequate. From now on, we assume that $\tau_C(x,y)<+\infty$.
The idea of the construction is the following. Let us first consider
the line through $x$ and $y$. We want $f$ to increase from $0$ at $x$
to $1$ at the boundary in the direction of $y$, in an affine way; and
to be equal to zero in the other direction.
% This idea can already be seen in Figure~\ref{fig:bound}.
Then, we will have to extend $f$ to all $C$ in a convex way.
Let $\Vec{u}=\tau_C(x,y)(y-x)$. For every $z\in C$, let us define $\sigma(z)=\sup_{}\left\{ t\geqslant 0\,|\,z+t\Vec{u}\in C \right\}$. $\sigma$ clearly takes values in $[0,+\infty]$. Consider the following function.
\[ \begin{array}{cccc}
\phi:& C &\longrightarrow& [-\infty,1] \\
& z &\longmapsto& 1-\sigma(z)
\end{array}. \]
Let us prove that $\phi$ is convex. Let $z_1$ and $z_2$ be two points in $C$ and $z_3=\lambda z_1+(1-\lambda)z_2$ (with $\lambda\in (0,1)$) a convex combination.
By definition of $\sigma$, if we take two real numbers $s_1$ and $s_2$ such that $0\leqslant s_1\leqslant \sigma(z_1)$ and $0\leqslant s_2\leqslant \sigma(z_2)$, we have:
\[ 
\begin{cases}
z_1+s_1\Vec{u}\in C\\
z_2+s_2\Vec{u}\in C.
\end{cases}
 \]
And thus, the convex combination of these two points with coefficients $\lambda$ and $1-\lambda$ also belongs to $C$:
\[ \lambda(z_1+s_1\Vec{u})+(1-\lambda)(z_2+s_2\Vec{u})\in C. \]
This point can be rewritten with $z_3$:
\[ z_3+\left( \lambda s_1+(1-\lambda)s_2 \right)\Vec{u}\in C. \]
By definition of $\sigma(z_3)$, we have $\lambda s_1+(1-\lambda)s_2\leqslant \sigma(z_3)$. This inequality is true for every $s_1\leqslant \sigma(z_1)$ and $s_2\leqslant z(s_2)$. Consequently:
\[ \lambda \sigma(z_1)+(1-\lambda)\sigma(z_2)\leqslant \sigma(z_3). \]
We can now prove the convexity inequality.
\begin{align*}
\phi(z_3)=1-\sigma(z_3)&\leqslant 1-\left( \lambda \sigma(z_1)+(1-\lambda)\sigma(z_2) \right)\\
&=\lambda(1-\sigma(z_1))+(1-\lambda)(1-\sigma(z_2))\\
&=\lambda \phi(z_1)+(1-\lambda)\phi(z_2).
\end{align*}

We now choose $f=\max_{}(\phi,0)$. Since $\phi\leqslant 1$, $f$ takes values in $[0,1]$. Let us prove that $f$ satisfies the desired equality. Let us compute $f(x)$ and $f(y)$. 
\begin{align*}
\sigma(x)&=\sup_{}\left\{ t\geqslant 0\,|\,x+t\Vec{u}\in C \right\}\\
&=\sup_{}\left\{ t\geqslant 0\,|\,x+t\tau_C(x,y)(y-x)\in C \right\} \\
&=\frac{1}{\tau_C(x,y)}\sup_{}\left\{ t'\geqslant 0\,|\,x+t'(y-x)\in C \right\}\\
&=\frac{1}{\tau_C(x,y)}\tau_C(x,y)\\
&=1.
\end{align*}
Thus $\phi(x)=1-\sigma(x)=0$ and $f(x)=\max_{}(0,0)=0$. Similarly, we
can prove:
\[ \sigma(y)=\frac{\tau_C(x,y)-1}{\tau_C(x,y)}, \]
% \begin{align*}
% \sigma(y)&=\sup_{}\left\{ t\geqslant 0\,|\,y+t\Vec{u}\in C right\}\\
% &=\sup_{}\left\{ t\geqslant 0\,|\,x+(t\tau_C(x,y)+1)(y-x)\in C \right\}\\
% &=\frac{1}{\tau_C(x,y)}\sup_{}\left\{ t'\geqslant 0\,|\,x+(t'+1)(y-x)\in C \right\}\\
% &=\frac{\tau_C(x,y)-1}{\tau_C(x,y)},
% \end{align*}
and thus, $\phi(y)=1-\sigma(y)=\tau_C(x,y)^{-1}$ and $f(y)=\max_{}(\tau_C(x,y)^{-1},0)=\tau_C(x,y)^{-1}$. We finally get:
\[ f(y)-f(x)=\frac{1}{\tau_C(x,y)}. \]
The construction of $g$ is analogous.
\end{proof}

\section{The Maximal Subdifferential}
\label{sec:impl-restr-subd}

In the case of a nonempty convex subset $C\subset\mathbb{R}^n$, and a given
point $x_0\in C$, we wonder what is the maximal subdifferential at $x_0$ (in
the sense of inclusion) for a function $f:C\longrightarrow
[m,M]$. We will prove that there \emph{is} a maximal one, and will
express it in terms of the subdifferential of a translation of the
Minkowski gauge. For each $x_0\in C$, we define $g_{C,x_0}:C\longrightarrow [0,1]$ by
\[ g_{C,x_0}(x)=\inf\left\{ \lambda>0\,|\,x-x_0\in \lambda(C-x_0)  \right\}.\]
% In this section, a function $f:C\longrightarrow
% \mathbb{R}$ will be implicitely extended to $\mathbb{R}^n$ by setting
% $f(x)=+\infty$ for $x\not\in C$.

% \begin{definition}
%   Let $C$ be a nonempty convex subset of $\mathbb{R}^n$ and $x_0\in
%   C$. We define \end{definition}
This function is obviously well-defined, and can be seen as a Minkowski
gauge centered in $x_0$ and restricted to $C$. It is well-known fact
that the Minkowski gauge is a convex function. So is this one.

\begin{theorem}
  Let $C$ be a nonempty convex subset of $\mathbb{R}^n$ and $x\in
  C$. We have
\[ \max_{\substack{f:C\longrightarrow [m,M]\\f\text{convex}}}\partial
f(x)=(M-m)\partial g_{C,x}(x), \]
where the maximum is understood in the sense of inclusion.
\end{theorem}
\begin{proof}
  Let us first relate $g_{C,x_0}$ to $\tau$. Let $x_0,x\in C$. We have
\begin{align*}
  g_{C,x_0}(x)&=\inf\left\{ \lambda>0\,|\,x-x_0\in \lambda(C-x_0)
  \right\}\\
  &=\sup\left\{ t>0\,|\, x-x_0\in \frac{1}{t}(C-x_0)  \right\}^{-1}\\
  &=\sup\left\{ t>0 \,|\, x_0+t(x-x_0)\in C  \right\}^{-1}\\
  &=\frac{1}{\tau(x_0,x)}.
\end{align*}

  Let us prove the result in the case $m=0$ and $M=1$, from which the
  general case follows immediately. Let $f:C\longrightarrow [0,1]$ be
  a convex function and $x_0\in C$. Let us show that $\partial
  f(x_0)\subset\partial g_{C,x_0}(x_0)$. This is true if $\partial
  f(x_0)$ is empty. Otherwise, let $\zeta\in \partial f(x_0)$. For
  every $x\in C$, we have
\begin{align*}
  \left< \zeta\middle| x-x_0\right> \leqslant  f(x)-f(x_0) &\leqslant\frac{1}{\tau(x_0,x)} \\
  &=g_{C,x_0}(x) =g_{C,x_0}(x)-g_{C,x_0}(x_0),
\end{align*}
where we used Theorem~\ref{thm:bound} for the second inequality. If
$x\not\in C$, the equality also holds, since $g_{C,x_0}(x)=+\infty$. We
thus have $\partial f(x_0)\subset \partial g_{C,x_0}(x_0)$. We
conclude by saying that $g_{C,x_0}$ is a convex function on $C$ with
values in $[0,1]$.
\end{proof}

\bibliography{variations}
\bibliographystyle{siam}
\bigskip
\bigskip
\bigskip

\begin{center}
\textsc{Joon Kwon\\[8pt]
Institut de math\'{e}matiques de Jussieu\\
\'{E}quipe combinatoire et optimisation\\
Universit\'{e} Pierre-et-Marie-Curie\\
4 place Jussieu\\
75252 Paris cedex 05 -- FRANCE\\[8pt]
e-mail: \textnormal{\texttt{joon.kwon@ens-lyon.org}}}
\end{center}

\end{document}